\documentclass[11pt]{amsart}
\usepackage{amsmath} \allowdisplaybreaks[4] 
\usepackage{epsfig}
\usepackage{color}
\usepackage{verbatim}
\usepackage{graphicx,amssymb,amsmath,amsthm}
\usepackage{mathrsfs}
\usepackage{hyperref} 

\usepackage{amssymb}
\usepackage{enumerate}
\newcommand{\vx}{\boldsymbol x}
\newcommand{\vh}{\boldsymbol h}
\newcommand{\ve}{\boldsymbol e}
\newcommand{\vy}{\boldsymbol y}
\newcommand{\vz}{\boldsymbol z}

\newcommand{\veta}{\boldsymbol \eta}

\newcommand{\dist}{{\rm dist}}
\newcommand{\cA}{\mathcal A}
\newcommand{\vA}{{\boldsymbol A}}
\newcommand{\vF}{\mathbb F}

\newcommand{\R}{{\mathbb R}}
\newcommand{\abs}[1]{\lvert#1\rvert}

\newtheorem{prop}{Proposition}[section]

\newtheorem{theorem}[prop]{Theorem}
\newtheorem{remark}[prop]{Remark}
\newtheorem{lem}[prop]{Lemma}
\newtheorem{defi}[prop]{Definition}
\newtheorem{corollary}[prop]{Corollary}

\usepackage{amsmath}
\allowdisplaybreaks[4]

\begin{document}
\bibliographystyle{plain}

\title { Instance optimality in phase retrieval }

\author{Yu Xia}
\thanks {Yu Xia was supported by NSFC grant (12271133, U21A20426, 11901143) and the key project of Zhejiang Provincial Natural Science Foundation grant (LZ23A010002)}
\address{Department of Mathematics, Hangzhou Normal University, Hangzhou 311121, China}
\email{yxia@hznu.edu.cn}



\author{Zhiqiang Xu}
\thanks {Zhiqiang Xu is supported  by
the National Science Fund for Distinguished Young Scholars (12025108), NSFC (12471361, 12021001, 12288201) and National Key R\&D Program of China (2023YFA1009401). }
\address{LSEC, ICMSEC, Academy of Mathematics and Systems Science, Chinese Academy of Sciences, Beijing 100190, China;\newline
School of Mathematical Sciences, University of Chinese Academy of Sciences, Beijing 100049, China}
\email{xuzq@lsec.cc.ac.cn}

\maketitle

\begin{abstract}
 Compressed sensing has demonstrated that a general signal $\vx \in \mathbb{F}^n$ ($\mathbb{F}\in \{\mathbb{R},\mathbb{C}\}$)
  can be estimated from few linear measurements with an error  {proportional to} the best $k$-term approximation error, a property known as instance optimality. 
    In this paper, we investigate instance optimality in the context of phaseless measurements using the $\ell_p$-minimization decoder, where $p \in (0, 1]$, for both real and complex cases. More specifically, we prove that $(2,1)$ and $(1,1)$-instance optimality of order $k$ can be achieved with $m = O(k \log(n/k))$ phaseless measurements, paralleling results from linear measurements. These results imply that one can stably recover approximately $k$-sparse signals from $m = O(k \log(n/k))$ phaseless measurements.
    Our approach leverages the phaseless bi-Lipschitz condition. 
   Additionally, we present a non-uniform version of $(2,2)$-instance optimality result in probability applicable to any fixed vector $\vx \in \mathbb{F}^n$. 
   These findings reveal striking parallels between compressive phase retrieval and classical compressed sensing, enhancing our understanding of both phase retrieval and instance optimality.

\end{abstract}
\section{introduction}

Assume that the target signal is $\vx=(x_1,\ldots,x_n)\in \vF^n$ where $\vF\in \{{\mathbb R}, {\mathbb C}\}$.
It is often assumed that $\vx$ belongs to or is close to the vectors in $k$-sparsity set:
 \begin{equation}
 \label{eqn: sigma_k}
 \Sigma_k:=\{\vz\in \vF^n: \|\vz\|_0\leq k\},
 \end{equation}
 where $k<n$.
  A widely used method for measuring the distance between $\vx$ and $\Sigma_k$
  is the {\em best $k$-term approximation error}, defined as
 \[
 \sigma_k(\vx)_q:=\min_{\vz\in \Sigma_k}\|\vz-\vx\|_q.
 \]
 Here, $q>0$ and $\|\cdot \|_q$   represents the  $\ell_q$ (quasi-)norm.
 
 We consider the encoder $F: \mathbb{F}^n \rightarrow \mathbb{R}^m $ as a mapping, where the information we gather about $\vx$ is represented by $ F({\vx}) $. A prominent topic in applied mathematics is the estimation of $\vx$ from $ F({\vx}) $. In particular, there is a keen interest in designing a decoder $ \Delta: \mathbb{R}^m \rightarrow \mathbb{F}^n $ such that $ \Delta(F(\vx))$ is as close to $\vx$ as possible.

We use $\cA_{n,m}$ to denote the set of encoder-decoder pairs $(F,\Delta)$. Following \cite{best},  for $q_1,q_2 >0$, we say that $(F,\Delta)\in \cA_{n,m}$ is {\em $(q_1,q_2)$-instance optimality of order $k$} if 
\begin{equation}\label{instance_def}
\|\vx-\Delta(F(\vx))\|_{q_1} \leq C k^{\frac{1}{q_1}-\frac{1}{q_2}}\sigma_k(\vx)_{q_2}\qquad \text{for all }\vx \in \vF^n.
\end{equation}
 Here, $C$ is a constant independent of $k$ and $n$. 
 The goal is to identify a pair $(F, \Delta) \in \mathcal{A}_{n,m}$ that fulfills equation (\ref{instance_def}) while minimizing the measurement number  $m$.

 Assume that  $\vA\in \vF^{m\times n}$  is the measurement matrix. 
 In compressed sensing, the scenario where $F(\vx) = \vA\vx$ represents linear measurements has been extensively studied  (see \cite{candes,best}).
  Additionally, instance optimality under nonlinear observations has been explored in recent research \cite{phase22, GWX, KG}. In this paper, we focus specifically on the case where $F({\vx}) = \abs{\vA\vx}$.
Here, the operator $\abs{\cdot}$ denotes the element-wise absolute value function applied to a vector, yielding another vector of the same dimension with non-negative elements.

The estimation of $\vx$ from $\abs{\vA\vx}$ is a well-established challenge in signal processing known as {\em phase retrieval}. This fundamental problem arises across a diverse array of scientific and engineering disciplines, including but not limited to X-ray crystallography, quantum tomography, electron microscopy, and audio signal processing. 

 {The primary objective of this paper is establishing the result of instance optimality to the domain of phaseless measurements, especially in the case where $\mathbb{F} = \mathbb{C}$, while simultaneously establishing the stable recovery of approximately $k$-sparse signals $\vx$ from noisy phaseless measurements.}

\subsection{Linear measurements }
To start, we will present the established results regarding instance optimality within the context of linear measurements.
 The $\ell_1$-minimization decoder, employed in the context of linear measurements and referred to as $\Delta^{\text{Linear}}_{1,\eta}$, is widely utilized and is defined as follows:
\[
\Delta^{\text{Linear}}_{1,\eta} (\vy):= \mathop{\rm argmin}\limits_{\vz\in \mathbb{F}^n} \{\|\vz\|_1 \ :\  \|\vA\vz -\vy\|_2\leq \eta\}.
\]
This is a numerically tractable decoder \cite{candes}. Stability results can be characterized through the concept of the  \emph{restricted isometry property} (RIP) of the matrix $\vA$. For $\delta\in (0,1)$ and a positive integer $k$, we say that the matrix $\vA \in {\vF}^{m \times n}$ satisfies the RIP$(k,\delta)$ if the following condition holds:
\begin{equation}\label{eq:rip}
(1-\delta)\|\vz\|_2\leq \|\vA\vz\|_2\leq (1+\delta)\|\vz\|_2,\quad \text {for all}\ \ \vz \in \Sigma_k.
\end{equation}
In particular, Cand\`es, Romberg, and  Tao \cite{candes} demonstrated that the pair $(\vA, \Delta^{\text{Linear}}_{1,\eta}) \in \mathcal{A}_{m,n}$ satisfies  the following inequality:
\begin{equation}\label{eq:ins210}
\|\Delta^{\text{Linear}}_{1,\eta}(\vA\vx+\ve)-\vx \|_2\leq C\cdot \frac{\sigma_k(\vx)_1}{\sqrt{k}}+D\cdot \eta,\quad \text {for all}\  \ \vx \in \vF^n,
\end{equation}
provided that  $\eta\geq \|\ve\|_2$ and that ${\vA}$ satisfies   RIP$(4k, \delta)$ with $\delta\in (0, 1/2)$. Here, $C$ and $D$ are universal positive constants depending on $\delta$. 

The stability result in (\ref{eq:ins210}) can likewise be leveraged to elucidate the behavior of $(2,1)$-instance optimality. More concretely, set $\Delta^{\text{Linear}}_1:=\Delta^{\text{Linear}}_{1,0}$, that is,
\[
\Delta^{\text{Linear}}_{1} (\vA \vx):= \mathop{\rm argmin}\limits_{\vz\in \mathbb{F}^n} \{\|\vz\|_1\ :\  \vA\vz =\vA \vx\}.
\]
Consequently,  (\ref{eq:ins210})  illustrates that the pair  $(\vA, \Delta^{\text{Linear}}_1)$ achieves $(2,1)$-instance optimality of order $k$, given that  $\vA$ satisfies  RIP$(4k, \delta)$ with $\delta\in (0,1/2)$.

The concept of instance optimality for linear measurements was systematically articulated by Cohen, Dahmen, and DeVore in \cite{best}, wherein they established the relationship between instance optimality and Gelfand width.  In \cite[Theorem 4.4]{best}, they further demonstrated that $({\vA}, \Delta^{\text{Linear}}_1)$
  satisfies $(1,1)$-instance optimality of order $k$ provided $\vA$ satisfies RIP($3k, \delta$) with $\delta\in (0,(\sqrt{2}-1)^2/3)$. 
  
    Certain families of random matrices, such as Gaussian random matrices, are known to satisfy the restricted isometry property RIP$(k, \delta)$ with high probability, provided that $m = O\left(k \log\left({en}/{k}\right) / \delta^2\right)$.
 Thus, by synthesizing this result with the aforementioned findings, we can deduce that for $m = O(k \log(en/k))$, there exists a pair $(\boldsymbol{A}, \Delta^{\text{Linear}}_1) \in \mathcal{A}_{m,n}$
that exhibits both $(2,1)$-instance optimality and $(1,1)$-instance optimality.

A more comprehensive approach for recovering the unknown sparse signal involves solving the following model $\Delta^{\text{Linear}}_{p,\eta}$ ($0<p\leq 1$):
 \begin{equation}\label{eqn: lp_linear}
\Delta^{\text{Linear}}_{p,\eta} (\vy)= \mathop{\rm argmin}\limits_{\vz\in \mathbb{F}^n} \{\|\vz\|_p \ :\  \|\vA\vz -\vy\|_2\leq \eta\}.
\end{equation}
A  numerical scheme for computing the solution to (\ref{eqn: lp_linear}) can be found in \cite{FoLai}. The stability result for $\Delta^{\text{Linear}}_{p,\eta}$ is detailed in \cite[Theorem 3.1]{FoLai}, which states:
\begin{equation}\label{eq:ins21}
\|\Delta^{\text{Linear}}_{p,\eta}(\vA\vx+\ve)-\vx \|_{q_1}\leq C\cdot \frac{\sigma_k(\vx)_{q_2}}{k^{1/q_2-1/q_1}}+D\cdot k^{1/q_1-1/2}\cdot \eta,\quad \text {for all}\  \ \vx \in \vF^n,
\end{equation}
provided that  $\eta\geq \|\ve\|_2$ and $\vA$ satisfies RIP$(2k+2,\delta)$ where $\delta$ is a constant dependent on $p$, as detailed in \cite{FoLai}.
 The pairs $(q_1,q_2)$  can be  either $(p,p)$ or $(2,p)$, where $0<p\leq 1$. Here $C$ and $D$ are  positive constants  depending on $\delta$ and $p$. 

The key advantage of $\Delta^{\text{Linear}}_{p,\eta}$ is its ability to achieve a better RIP constant $\delta$ compared to  the model $\Delta^{\text{Linear}}_{1,\eta}$.  Optimal bounds for the  RIP constant $\delta$ in relation to $\Delta^{\text{Linear}}_{p,\eta}$ are discussed in \cite{Zhang}. Additionally,  when $\eta=0$, the results indicate  $(p,p)$-instance optimality and $(2,p)$-instance optimality. Notably, when $p=1$, this reduces to showing $(1,1)$-instance optimality and $(2,1)$-instance optimality for $\Delta_1^{\text{Linear}}$.

We will next examine   $(2,2)$-instance optimality for $\Delta^{\text{Linear}}_{1}$.  
In fact, as shown in \cite{best}, even for the case $k=1$, to achieve $(2,2)$-instance optimality for all $\vx\in \mathbb{F}^n$, the measurement number $m\geq c_0n$ for some universal constant $c_0>0$. Hence,  one has turned to  consider the non-uniform version of $(2,2)$-instance optimality to reduce the measurement number \cite{best}.
In contrast to the requirement that (\ref{instance_def}) must  hold for all $\vx\in \mathbb{F}^n$, the non-uniform $(2,2)$-instance optimality is considered for each specific $\vx_0\in \mathbb{F}^n$ \emph{in probability}.


 More concretely, the results are presented in the following form: for  each fixed $\vx_0\in \mathbb{F}^n$, if $\vA\in \mathbb{F}^{m\times n}$ is drawn randomly, then 
\begin{equation}\label{eqn: (2,2)-instance}
\|\vx_0-\Delta^{\text{Linear}}_{1}(\vA\vx_0)\|_2\leq C\sigma_k(\vx_0)_2
\end{equation}
holds for this particular $\vx_0$ with high probability.

Notably, when $\vA\in \mathbb{F}^{m\times n}$
  is a Gaussian random matrix satisfying $k \leq \alpha m/\log(en/k)$ for some universal positive constant $\alpha$, it has been demonstrated that equation (\ref{eqn: (2,2)-instance}) holds with high probability \cite{L1221, L1222}.

\subsection{Phaseless measurements} 
In this paper, we focus on the phaseless encoder $F(\vx)=\abs{\vA\vx}$. A fundamental observation is that $\abs{\vA \vx}=\abs{\vA (c\vx)}$ for any unimodular constant $c\in \vF$ (i.e., $\abs{c}=1$).
Consequently, in the context of phase retrieval, it is conventional to regard a vector $\vx$ and its scalar multiple $c\vx$ as equivalent, where $c\in \vF$ is unimodular. This equivalence reflects the inherent phase ambiguity in the problem.
To account for this ambiguity, we measure the distance between vectors $\vx$ and $\vy$ in the following manner:
For any $\boldsymbol{x},\boldsymbol{y}\in \mathbb{F}^{n}$ , we define the distance between $\boldsymbol{x}$ and $\boldsymbol{y}$  in $\ell_p$ (quasi-)norm as
\[
\text{dist}_p(\vx,\vy):=\text{dist}_{\mathbb{F},p}(\boldsymbol{x},\boldsymbol{y})=\underset{c\in \mathbb{F}, \abs{c}=1}{\min}\|\boldsymbol{x}-c\boldsymbol{y}\|_p.
\]
 We omit the subscript $p$ in $\dist_p$  when $p = 2$, i.e., $\text{dist}(\vx,\vy):=\text{dist}_2(\vx,\vy)$. As said before, in the phase retrieval framework, we consider $\vx\sim \vy$ if and only if $\dist(\vx,\vy)=0$.
Throughout the rest of this paper, we adopt a slightly unconventional notation by representing the set $\{c\cdot \vx : \abs{c}=1, c \in \vF \}$ simply as the vector $\vx \in \mathbb{F}^n$.   For convenience, we use $\abs{\vA}$ to denote the phaseless encoder corresponding to the matrix $\vA$. 

The phaseless $\ell_1$-minimization decoder, denoted as $\Delta_1$, is defined as follows: 
  \begin{equation}\label{eq:phaselessl1}
 \Delta_1(|\vA \vx|):=\mathop{\rm argmin}\limits_ {\vz\in {\mathbb F}^n}\{\|\vz\|_1: \abs{\vA\vz}=|\vA\vx|\}.
 \end{equation}
 {Despite the lack of theoretical convergence analysis, several efficient algorithms, such as the iterative projection algorithm \cite{Moravec} and ADM algorithm \cite{Yang}, have been developed to solve (\ref{eq:phaselessl1}).} Here, we modify the definition of instance optimality in phase retrieval problem as follows:
  For $q_1,q_2 >0$, we say that $(\abs{\vA},\Delta_1)$ is {\em phaseless $(q_1,q_2)$-instance optimality of order $k$} if 
\[
\dist_{q_1}{(\vx, \Delta_1(\abs{\vA\vx}))} \leq C\cdot k^{\frac{1}{q_1}-\frac{1}{q_2}}\sigma_k(\vx)_{q_2},\qquad \text{for all }\vx \in \vF^n.
\]
 Here, $C$ is a constant independent of $k, n$ and 
 \[
 \sigma_k(\vx)_q:=\min\limits_{{\vz\in \Sigma_k}}{\rm  dist}_q(\vx,\vz).
 \]

In the case where $\mathbb{F} = \mathbb{R}$, phaseless instance optimality was investigated in \cite{GWX}. 
Theorem 3.1 and Theorem 4.4 in \cite{GWX} demonstrate that $(\abs{\boldsymbol{A}}, \Delta_1)$  exhibits phaseless $(2, 1)$-instance optimality and $(1, 1)$-instance optimality of order $k$, respectively, provided $\boldsymbol{A}\in \R^{m\times d}$  is a standard Gaussian random matrix with $m = O(k \log(en/k))$.
The analytical tool employed in \cite{GWX} is the strong restricted isometry property, initially introduced by Voroninski and Xu \cite{VX}. However, this approach does not readily extend to cases where  $\mathbb{F} = \mathbb{C}$.

   {Xia and Xu \cite{XiaXu} proved that for $\mathbb{F} = \mathbb{C}$, $\Delta_1(|\vA\vx|) = \vx$
  holds with high probability for all $\vx \in \Sigma_k$, when $\vA \in \mathbb{C}^{m \times n}$
  is a Gaussian random matrix with $m = O(k \log(en/k))$. However, the instance optimality of the $\Delta_1$
  decoder in (\ref{eq:phaselessl1}) for arbitrary $\vx$ remains unresolved.\\
\indent While algorithms such as ThWF \cite{ThWF} and HTP \cite{Cai_SPR} efficiently recover $\vx \in \Sigma_k$ from $|\vA\vx|$, their worst-case sampling complexity of $O(k^2 \log n)$ is suboptimal. This suboptimality arises in part from the challenge of accurately identifying the support of sparse signals, a problem closely related to sparse PCA,  where polynomial-time recovery is widely believed to be impossible unless $k \lesssim \sqrt{n}$  \cite{PCA}. Furthermore, these algorithms neither address recovery errors for general $\vx \in \mathbb{F}^n$ nor establish phaseless $(q_1, q_2)$-instance optimality, thus leaving a theoretical gap in understanding their performance in more general settings.
}
 
 \subsection{Our contribution}

 In this paper, we focus on investigating the performance of the pair $(\abs{\boldsymbol{A}}, \Delta_1)$, especially in the context where $\mathbb{F} = \mathbb{C}$. The following questions arise of particular interest especially when $\vx\in \mathbb{C}^n$ and $\vA\in \mathbb{C}^{m\times n}$:
\begin{enumerate}[\bf{Question} i: ]
\item For general signal $\vx\in \mathbb{C}^n$  with a small $\sigma_k(\vx)_1$, is it possible to obtain a stable estimate of $\vx$ from $\abs{\vA\vx}\in \R^m$?

\item 
Consider $\vA\in \mathbb{C}^{m\times n}$  as a standard complex Gaussian random matrix with $m  \gtrsim k\log(en/k)$. 
Does the pair $(|\boldsymbol{A}|, \Delta_1)$
exhibit $(1, 1)$-instance optimality and $(2, 1)$-instance optimality? 

\item
Consider the measurement matrix $\vA$ as described in Question ii.  For any fixed $\vx_0\in \mathbb{C}^n$, does the pair $(\abs{\boldsymbol{A}}, \Delta_1)$
  exhibits $(2, 2)$-instance optimality of order $k$ with high probability?

\end{enumerate}

These questions focus on the potential for stable signal recovery and the optimality of the $\ell_1$-minimization decoder within the framework of phaseless compressed sensing. 
  Our aim is to address these inquiries comprehensively. In fact, this paper expands the analysis to encompass a more general $\ell_p$-minimization decoder for $p\in  (0,1]$, which includes the $\ell_1$ case as a specific instance.
  
  We tackle Question i by utilizing the phaseless bi-Lipschitz property of phase retrieval, as discussed in \cite{XXX}. 
  The key aspect of our solution to Question i is detailed in Corollary \ref{th: main}, which is a specific case of Theorem \ref{th:main11}.

Questions ii and iii are explored in the context of standard Gaussian random measurement matrix $\vA\in \mathbb{F}^{m\times n}$ when $\mathbb{F}\in \{\mathbb{R},\mathbb{C}\}$, which is asymptotically optimal regarding the condition number of the phaseless bi-Lipschitz property \cite{XXX}. When $m = O(k \log(en/k))$, we show that both $(1, 1)$-instance optimality and $(2, 1)$-instance optimality of order $k$ can be achieved for the pair $(|\boldsymbol{A}|, \Delta_1)$, as noted in Corollary \ref{(2,1) and (1,1)}. This resolves Question ii.

  {
Furthermore, Proposition \ref{thm:bad_22} demonstrates that $(2,2)$-instance optimality is not a viable concept for phaseless measurements when $m \ll n$. This limitation motivates us to address Question iii by analyzing non-uniform recovery results. In Theorem \ref{th:22}, we establish that for any fixed $\boldsymbol{x}_0 \in \mathbb{C}^n$, non-uniform $(2, 2)$-instance optimality of order $k$ holds for $(|\boldsymbol{A}|, \Delta_1)$, thereby providing a positive answer to Question iii.}

\subsection{Organization}
The organization of this paper is organized as follows. Section \ref{sec2} rigorously examines the instance optimality of the phase retrieval problem in both real and complex contexts, systematically addressing Questions i, ii, and iii. Section \ref{proof of Thm3} presents a comprehensive proof of Theorem \ref{th:main11}, which is pivotal for Corollary \ref{th: main} and serves to resolve Question i. Section \ref{sec: Gaussian} articulates the proof of Theorem \ref{th:lip}, which is directly relevant to Corollary \ref{(2,1) and (1,1)} and addresses Question ii.  {Finally, Section  \ref{sec: 22inst_new} and Section \ref{sec: 22inst2} elucidate the proofs of Proposition \ref{thm:bad_22} and Theorem \ref{th:22} repsectively, thereby conclusively addressing Question iii.}


\section{Main Results}\label{sec2}

To investigate phaseless instance-optimality in complex case, we introduce the phaseless bi-Lipschitz condition for the phaseless measurements when  $\mathbb{F}\in \{\mathbb{R},\mathbb{C}\}$.  The phaseless bi-Lipschitz condition bears a close resemblance to the restricted isometry property used in linear measurements, rendering it particularly suitable for phaseless measurements in both real and complex domains.

\begin{defi}(bi-Lipschitz) 
Assume that $\mathcal{X}\subset{\mathbb{F}}^{n}$.
We say that  $\vA\in \mathbb{F}^{m\times n}$ satisfy the phaseless bi-Lipschitz condition on the set $\mathcal{X}$ with positive constants $L$ and $U$ if
\begin{equation}\label{bi_constraint}
L\cdot \mathrm{dist}(\boldsymbol{x},\boldsymbol{y})\leq \||{\vA}\boldsymbol{x}|-|{\vA}\boldsymbol{y}|\|_2\leq U\cdot \mathrm{dist}(\boldsymbol{x},\boldsymbol{y}),\qquad \text {for all } \boldsymbol{x},\boldsymbol{y} \in \mathcal{X}.
\end{equation}
\end{defi}
 \begin{remark}
  {
 The phaseless bi-Lipschitz condition plays a key role in our analysis by extending RIP for compressed sensing  \cite{FoLai} from linear to phaseless measurements. 
 Unlike the Strong Restricted Isometry Property (SRIP), which views phaseless measurements as the union of finite linear measurements and is  applicable only to real case \cite{GWX,VX}, the phaseless bi-Lipschitz  condition is more versatile, handling both real and complex phaseless measurements.
  A key technical challenge in our analysis is estimating the term $\||\vA\vx| - |\vA\vy|\|_2$ when $\vx \notin \Sigma_k$ and $\vy \in \Sigma_k$.
  To do that, we utilize Lemma \ref{temp_lem}, which yields, for an appropriately chosen $\vz\in \Sigma_k$, 
  \[
  \||\vA\vz| - |\vA\vy|\|_2 - \|\vA(\vx-\vz)\|_2 \leq \left\| |\vA\vx| - |\vA\vy| \right\|_2 \leq \||\vA\vz| - |\vA\vy|\|_2 + \|\vA(\vx-\vz)\|_2.
  \]
  This decomposition allows us to estimate $\||\vA\vz| - |\vA\vy|\|_2$ using the phaseless bi-Lipschitz condition (since both $\vz, \vy \in \Sigma_k$) and to estimate $\|\vA(\vx-\vz)\|_2$ by leveraging techniques from compressed sensing. 
  This combined estimation strategy then enables us to estimate $\||\vA\vx| - |\vA\vy|\|_2$,
  which is a key step for establishing phaseless instance optimality.}
 \end{remark}

To achieve stable recovery in the phase retrieval problem, analogous to $\Delta^{\text{Linear}}_{p,\eta}$ in the linear measurements model, we define $\Delta_{p,\eta} : \mathbb{R}^m \rightarrow \mathbb{F}^n$ as the phaseless $\ell_p$-minimization decoder (for $p \in (0, 1]$), which is subject to $\ell_2$-bounded noise. Specifically, this is expressed as:
  \begin{equation}\label{eq:phaselessp}
 \Delta_{p,\eta}(\vy):=\mathop{\rm argmin}\limits_ {\vz\in {\mathbb F}^n}\{\|\vz\|_p\ :\ \|\abs{\vA\vz}-\vy\|_2\leq \eta\}.
 \end{equation}
 
By extending these concepts to phaseless measurements, we will subsequently demonstrate that the $\ell_p$-minimization decoder $\Delta_{p,\eta}$ (for $0 < p \leq 1$) achieves $(p, p)$-instance optimality and $(2, p)$-instance optimality, provided that $\boldsymbol{A}$ satisfies the phaseless bi-Lipschitz condition. While stability results for linear measurements can be found in \cite{FoLai,candes}, the non-linear nature of the phase retrieval problem renders the proof techniques employed in \cite{FoLai,candes} inapplicable to the results derived under model (\ref{eq:phaselessp}).

\begin{theorem}\label{th:main11}
Let $k \in \{1,\ldots,n\}$ be an integer and $\vA \in \mathbb{F}^{m\times n}$
  satisfy the phaseless bi-Lipschitz condition on the set $\mathcal{X}=\{\boldsymbol{x}\in \mathbb{F}^n\ :\ \|\boldsymbol{x}\|_0\leq (r+4)k\}$
 with positive constants $L$ and $U$.
The constant $r$ can be any positive real number that satisfies the inequality:
\begin{equation}\label{r_constraint}
L-U\cdot2^{1/p-1}\cdot\left(\frac{{2}}{{r}}\right)^{1/p-1/2}>0.
\end{equation}
 Then, for $p\in (0,1]$,
 the following holds for all $\vx\in \vF^n$:
 \begin{align}
 &{\rm dist}_p(\Delta_{p,\eta}(|\vA\vx|+\ve),{\boldsymbol{x}})\leq C_1 \cdot \sigma_{k}(\vx)_p+D_1 \cdot k^{1/p-1/2}\cdot \eta,\\
 &{\rm dist}(\Delta_{p,\eta}(|\vA\vx|+\ve),{\boldsymbol{x}})\leq  C_2\cdot  \frac{\sigma_{k}(\vx)_p}{k^{1/p-1/2}}+D_2\cdot \eta,
 \end{align}
 where  $\Delta_{p,\eta}$ is defined in (\ref{eq:phaselessp}) and $\|\ve\|_2\leq \eta$.
Here, the positive constants $C_1$, $C_2$, $D_1$, and $D_2$ depend on  the parameters $U$, $L$, $r$, and $p$,
with their explicit expressions detailed in the proof.
\end{theorem}
\begin{proof}
The proof can be found in Section \ref{proof of Thm3}.
\end{proof}

\begin{remark}
 {For complex Gaussian measurements, it was shown in \cite{XiaZhou} that the $\ell_p$-minimization decoder can achieve exact recovery of $k$-sparse complex signals under the condition $m \gtrsim k + pk \log(en/k)$. Furthermore, \cite{XiaZhou} proposes an iteratively reweighted algorithm for the $\ell_p$-minimization decoder. The convergence analysis of the associated algorithm for solving $\ell_p$ minimization in the presence of noise remains an open problem for future investigation.}
\end{remark} 
By setting $p=1$ and $\mathbb{F}=\mathbb{C}$, we can directly address Question i in Corollary \ref{th: main}.
\begin{corollary}\label{th: main}
Let $k \in \{1,\ldots,n\}$ be an integer and $\vA \in \mathbb{F}^{m\times n}$
  satisfy the phaseless bi-Lipschitz condition on the set $\mathcal{X}:=\{\boldsymbol{x}\in \mathbb{F}^n\ :\ \|\boldsymbol{x}\|_0\leq (r+4)k\}$
 with positive constants $L$ and $U$.
The constant $r$ can be any positive real number that satisfying  $r>2\cdot (U/L)^2$. Then for all $\vx\in \vF^n$, the following holds:
 \begin{align}
 &{\rm dist}_1(\Delta_{1,\eta}(|\vA\vx|+\ve),{\boldsymbol{x}})\leq C_1\cdot  \sigma_{k}(\vx)_1+D_1\cdot  \sqrt{k}\cdot \eta,\label{eqn: a}\\
 &{\rm dist}(\Delta_{1,\eta}(|\vA\vx|+\ve),{\boldsymbol{x}})\leq  C_2\cdot  \frac{\sigma_{k}(\vx)_1}{\sqrt{k}}+D_2\cdot \eta,\label{eqn: b}
 \end{align}
where  $\Delta_{1,\eta}$ is defined in (\ref{eq:phaselessp}) and $\|\ve\|_2\leq \eta$.
Here 
\[
\begin{aligned}
&C_1=\frac{2\cdot\frac{U}{L}\cdot\left(\frac{1}{r^{1/2}}+1\right)\cdot (2+r)^{1/2}}{1-\frac{U}{L}\left(\frac{{2}}{{r}}\right)^{1/2}}+2,\qquad D_1=  \frac{1}{L}\cdot\frac{2\cdot {(2+r)}^{1/2}}{1-\frac{U}{L}\cdot\left(\frac{{2}}{{r}}\right)^{1/2}},\\
& C_2=\frac{\frac{U}{L}\cdot\left(\frac{1}{r^{1/2}}+1\right)\cdot \Big(1+\frac{1}{(r/2)^{1/2}}\Big)}{1-\frac{U}{L}\cdot \left(\frac{{2}}{{r}}\right)^{1/2}}+\frac{1}{r^{1/2}}+1,\qquad D_2= \frac{1}{L}\cdot\frac{2\cdot\Big(1+\frac{1}{(r/2)^{1/2}}\Big)}{1-\frac{U}{L}\cdot\left(\frac{{2}}{{r}}\right)^{1/2}}.
 \end{aligned}
 \]
\end{corollary}
  {
 \begin{remark}
In phase retrieval problem, the lower bound in (\ref{bi_constraint}) is equivalent to the lower restricted isometry property (LRIP) \cite[Definition 2]{KG}. Adapting Theorem 1 from \cite{KG} to the sparse phase retrieval problem, we can find that if the decoder $\Delta$ is defined as:
\[
 \Delta(\vy):=\underset{\vz\in \mathbb{F}^{n}}{\mathrm{argmin}}\{\||\vA\vz|-\vy\|_2\ :\  \vz\in \mathcal{X}\},
\]
and if LRIP holds, then for all $\vx\in \mathbb{F}^{n}$, 
\begin{equation}
 \label{eqn: LRIP_error_estimation} 
 \mathrm{dist}(\vx,\Delta(|\vA\vx|+\ve))\lesssim \inf_{\vx'\in \mathcal{X}}\big(\mathrm{dist}(\vx,\vx')+ \||\vA\vx|-|\vA\vx'|\|_2\big)+ \|\ve\|_2.
  \end{equation} 
  Leveraging the upper bound in  (\ref{bi_constraint})  and employing similar techniques presented in this paper, we can establish a comparable phaseless instance optimality based on (\ref{eqn: LRIP_error_estimation}), analogous to that in (\ref{eqn: b}). For brevity, we omit the detailed derivation.
 \end{remark}
 }

Theorem \ref{th:main11} stipulates that the matrix $\boldsymbol{A} \in \mathbb{F}^{m \times n}$
  must satisfy the phaseless bi-Lipschitz condition as specified in the theorem. We now demonstrate that a standard Gaussian random matrix $\boldsymbol{A} \in \mathbb{F}^{m \times n}$
  fulfills this condition. The characterization of such a matrix is as follows:

\begin{enumerate}[1:]
\item For $\vF=\mathbb{R}$, the elements of $\boldsymbol{A}$ are independently drawn from $\mathcal{N}(0,1)$;
\item For $\vF=\mathbb{C}$, the elements of $\boldsymbol{A}$ are independently drawn from $\mathcal{N}(0,1/2)+\mathrm{i}\mathcal{N}(0,1/2)$.
\end{enumerate}
Consistent with \cite{XXX}, we define the constant $\beta_0^{\vF}$  as follows:
\begin{equation}\label{eq:mdlower-new}
\beta_0^{\vF}:=\begin{cases}
\sqrt{\frac{\pi}{\pi-2}}\,\,\approx\,\, 1.659,  & \text{if $\vF=\mathbb{R}$;}\\
\sqrt{\frac{4}{4-\pi}}\,\,\approx\,\, 2.159,  & \text{if $\vF=\mathbb{C}$.}\\
\end{cases}
\end{equation}
Then we have
\begin{theorem}\label{th:lip}
Let $k \in \{1,\ldots,n\}$ be an integer and
 $\vA\in \vF^{m\times n}$ be a standard Gaussian random matrix with $m
\geq C k \log (en/k)$. 
Then  with probability at least $1-\exp(-c_0m)$, the matrix $\vA\in \mathbb{F}^{m\times n}$ satisfies the phaseless bi-Lipschitz condition on the set $\mathcal{X}:=\{\boldsymbol{x}\in \mathbb{F}^n\ :\ \|\boldsymbol{x}\|_0\leq (r+4)k\}$ with positive constants $L$ and $U$ such that $\beta_0^{\mathbb{F}}\leq U/L\leq \beta_0^{\mathbb{F}}+0.01$. Here, $C>0$ is a constant that depends on $r$,  where $r$ is a  constant satisfying 
\begin{equation}\label{r_gaussian}
1-(\beta_0^{\mathbb{F}}+0.01)\cdot2^{1/p-1}\cdot\left(\frac{{2}}{{r}}\right)^{1/p-1/2}>0
\end{equation}
for any fixed $p\in(0,1]$,  and $c_0$ is a universal positive constant.
\end{theorem}
\begin{proof}
The proof can be found in Section \ref{sec: Gaussian}. 
\end{proof}
\begin{remark}\label{r_rem}
The condition in (\ref{r_gaussian}) implies the requirements specified in (\ref{r_constraint}).
 We can select a particular value of $r$ to fulfill the condition stipulated in (\ref{r_gaussian}). For instance, by opting for $r = 10$, the following inequality is valid for any $p \in (0, 1]$:
\[
\small
\begin{split}
&1-(\beta_0^{\mathbb{F}}+0.01)\cdot2^{1/p-1}\cdot\left(\frac{{2}}{{r}}\right)^{1/p-1/2}= 1 -(\beta_0^{\mathbb{F}}+0.01) \cdot \left(\frac{4}{r}\right)^{\frac{1}{p}-1} \cdot \left(\frac{2}{r}\right)^{\frac{1}{2}} \\
&=1 - (\beta_0^{\mathbb{F}} + 0.01) \cdot \left(\frac{4}{10}\right)^{\frac{1}{p}-1} \cdot \left(\frac{2}{10}\right)^{\frac{1}{2}}
 > 1 - (\beta_0^{\mathbb{F}} + 0.01) \cdot \left(\frac{1}{5}\right)^{\frac{1}{2}} > 0.
\end{split}\]
\end{remark}
By combining Corollary \ref{th: main} with Theorem \ref{th:lip}, we can directly obtain Corollary \ref{(2,1) and (1,1)}, which answers Question ii, when $\mathbb{F}=\mathbb{C}$.
\begin{corollary}\label{(2,1) and (1,1)}
 Let $\Delta_1$ be the phaseless $\ell_1$-minimization decoder as defined in (\ref{eq:phaselessl1}). Consider $\boldsymbol{A} \in \mathbb{F}^{m \times n}$ as a standard Gaussian random matrix and let $k \in \{1,\ldots,n\}$ be an integer. Then, if $m \geq C_0 k \log(en/k)$, the pair $(|\boldsymbol{A}|, \Delta_1)$ exhibits both $(2,1)$-instance optimality and $(1,1)$-instance optimality of order $k$ with probability at least $1 - \exp(-c_0 m)$. Here, $C_0$ and $c_0$ are two universal positive constants.
   \end{corollary}
   \begin{proof}
Based on Corollary \ref{th: main}, when $\eta = 0$, the results for $(1,1)$-instance optimality and $(2,1)$-instance optimality can be expressed as follows:
\[{\rm dist}_1(\Delta_{1,\eta}(|\boldsymbol{A}\boldsymbol{x}|), {\boldsymbol{x}}) \leq C_1 \cdot \sigma_k(\boldsymbol{x})_1 \quad \text{and} \quad {\rm dist}(\Delta_{1,\eta}(|\boldsymbol{A}\boldsymbol{x}|), {\boldsymbol{x}}) \leq C_2 \cdot \frac{\sigma_k(\boldsymbol{x})_1}{\sqrt{k}},\]
where $C_1$ and $C_2$ are positive constants dependent on $r$ and $U/L$, provided that $r > 2 \cdot (U/L)^2$.

 Thus, it suffices to identify an appropriate choice for $r$. According to Theorem \ref{th:lip} and Remark \ref{r_rem}, by selecting $r = 10$, we can assert that, with a probability of at least $1 - \exp(-c_0 m)$, the standard Gaussian random matrix satisfies the phaseless bi-Lipschitz condition with $ \beta_0^{\mathbb{F}} \leq U/L \leq \beta_0^{\mathbb{F}} + 0.01$ on the set 
 \[
 \mathcal{X} := \{\boldsymbol{x} \in \mathbb{F}^n \ : \ \|\boldsymbol{x}\|_0 \leq (r + 4)k\},
  \]
   provided $m \geq C_0 k \log(en/k)$ for some universal positive constant $C_0$.
{We can also verify that $r=10$ satisfies the condition $r >2 (\beta_0^{\mathbb{F}} + 0.01)^2\geq  2 \cdot (U/L)^2$ in Corollary \ref{th: main}.}
  In this scenario, $C_1$ and $C_2$ in Corollary \ref{th: main} can be bounded above by some universal positive constant. Thus, the proof is complete.
   \end{proof}
   \begin{remark}
The conclusion in  Corollary \ref{(2,1) and (1,1)} can be readily extended to the $\ell_p$-minimization decoder $\Delta_{p,0}$ defined in (\ref{eq:phaselessp}) for any fixed $0 < p \leq 1$. Consider $\boldsymbol{A} \in \mathbb{F}^{m \times n}$ as a standard Gaussian random matrix. By selecting $r = 10$ and ensuring that $m \gtrsim k \log(en/k)$, it follows from Theorem \ref{th:main11} and Theorem \ref{th:lip} that the pair $(|\boldsymbol{A}|, \Delta_{p,0})$ exhibits both $(2,p)$-instance optimality and $(p,p)$-instance optimality of order $k$ with high probability. Since the proof closely parallels the preceding arguments, we omit the details here.
   \end{remark}
  {We next turn to Question iii. To begin with, in the context of compressed sensing, it has been established that $(2,2)$-instance optimality is unattainable for $\vA \in \mathbb{F}^{m \times n}$
  when $m \ll n$ \cite{Gribonval1,best}. Proposition \ref{thm:bad_22} extends this result, demonstrating that such a limitation also applies to the  phase retrieval problem.
  \begin{prop}\label{thm:bad_22}
For any given matrix $\vA \in \mathbb{F}^{m\times n}$ ($m < n$) and integer $k \in \{1,\ldots,n\}$, suppose there exists a decoder $\Delta$ such that for all $\vx \in \mathbb{F}^n$, \begin{equation}
 \label{eqn: uniform_22_upper} 
 \mathrm{dist}(\Delta(|\vA\vx|),\vx)_2 \leq c_0 \cdot \sigma_k(\vx)_2, 
 \end{equation}
  where $c_0$
  is a positive absolute constant. Then $m \gtrsim n$.
  \end{prop}
\begin{proof}
The proof can be found in Section \ref{sec: 22inst_new}. 
\end{proof}}

However, for any fixed $\boldsymbol{x}_0 \in \mathbb{F}^n$, if $\boldsymbol{A} \in \mathbb{F}^{m \times n}$
  is drawn randomly, we can establish that \begin{equation} {\rm dist}(\Delta_1(|\boldsymbol{A}\boldsymbol{x}_0|), \boldsymbol{x}_0) \lesssim \sigma_k(\boldsymbol{x}_0)_2 \end{equation} holds for this particular $\boldsymbol{x}_0$  with high probability, if  $m\gtrsim k\log(en/k)$. This result addresses Question iii. Notably, similar results have been obtained for linear measurements in \cite {L1222,L1221}.

\begin{theorem}\label{th:22}
Let $\boldsymbol{x}_0 \in \mathbb{F}^n$ be an fixed vector. Assume that $\boldsymbol{A} \in \mathbb{F}^{m \times n}$ is a standard Gaussian random matrix. Let $k \in \{1,\ldots,n\}$ be an integer such that $C k \log\left({en}/{k}\right) \leq m$ for some universal positive constant $C$. The following result holds with probability at least
$1 - \widetilde{c}_3 \exp(-\widetilde{c}_2 m) - \exp(-\sqrt{mn}) - 2m \exp\left(-\frac{\widetilde{c}_1 m}{\log(en/m)}\right):$
\[\mathrm{dist}(\Delta_1(|\boldsymbol{A} \boldsymbol{x}_0|), \boldsymbol{x}_0) \lesssim \sigma_k(\boldsymbol{x}_0)_2.\]
Here, $\widetilde{c}_1$, $\widetilde{c}_2$, and $\widetilde{c}_3$ are also universal positive constants.
\end{theorem}
\begin{proof}
The proof can be found in Section \ref{sec: 22inst2}. 
\end{proof}

\begin{remark}
All the aforementioned results are related to phaseless measurements. In  \cite{group}, a nonlinear encoder referred to as {\em max filters} was introduced, which includes phaseless encoders as a special case. The bi-Lipschitz properties of max filters have been thoroughly investigated in \cite{balan}. Extending the above results to the max filters encoder would be an intriguing endeavor. We suggest this as a promising direction for future research.
\end{remark}

\section{Proof of Theorem \ref{th:main11}  }\label{proof of Thm3}
The following lemma is essential to our argument. We will postpone its proof until the end of this section.
\begin{lem}\label{temp_lem}
Let $\boldsymbol{x},\boldsymbol{y},\boldsymbol{z}\in \mathbb{F}^n$. Then the following inequality holds: 
\begin{equation}\label{eq:budeng}
\||\boldsymbol{x}|-|\boldsymbol{z}|\|_2-\|\boldsymbol{y}\|_2\leq \left\| \abs{\boldsymbol{x}+\boldsymbol{y}}-|\boldsymbol{z}|\right\|_2\leq \||\boldsymbol{x}|-|\boldsymbol{z}|\|_2+\|\boldsymbol{y}\|_2. 
\end{equation}
\end{lem}

We next present the proof of  Theorem \ref{th:main11}.

\begin{proof}[Proof of Theorem \ref{th:main11}]

Assume that $T_0$ denotes the index set of the $k$ largest elements of $\boldsymbol{x}$ in magnitude. Similarly, $T_1$ represents the index set of the $k$ largest elements of $\boldsymbol{x}_{T_0^c}$ in magnitude. This pattern can be extended to subsequent index sets $T_2, T_3, \ldots$. For convenience, define $T_{01} := T_0 \cup T_1$ and let $\boldsymbol{x}^\# := \Delta_{p}(|\boldsymbol{A}\vx|+\ve)$. Without loss of generality, it is assumed that
\[\langle \boldsymbol{x}^\#, \boldsymbol{x}_{T_{01}} \rangle \geq 0.\]

For convenience, denote 
\[\boldsymbol{h} := \boldsymbol{x}^\# - \boldsymbol{x}_{T_{01}}.\] Let $S_0 := T_{01} = T_0 \cup T_1$. Define $S_1$ as the index set of the $rk$ largest elements of $\boldsymbol{h}_{S_0^c}$ in magnitude, and extend this process iteratively to define $S_2, S_3, \ldots$ accordingly, where $r$ satisfies the condition (\ref{r_constraint}). Let $S_{01} := S_0 \cup S_1$.

We now proceed to estimate the upper bounds of $\text{dist}_p(\boldsymbol{x}^{\#}, \boldsymbol{x})$ for $p\in (0,1]$ and $\text{dist}(\boldsymbol{x}^{\#}, \boldsymbol{x})$. The ensuing proof is structured into four distinct steps.

\vspace{0.1cm}
{\textbf{Step 1: 
Establish  upper bounds for $(\text{dist}_p(\boldsymbol{x}^{\#}, \boldsymbol{x}))^p$ using $\|\vh_{S_{01}}\|_p^p$ and $\|\vh_{S_{01}^c}\|_p^p$, and for $\text{dist}(\boldsymbol{x}^{\#}, \boldsymbol{x})$ using $\|\vh_{S_{01}}\|_2$ and $\|\vh_{S_{01}^c}\|_2$, respectively}}.

Utilizing the definition of $\boldsymbol{h}$ in relation to the index sets $S_{01}$ and $T_{01}$, and leveraging the inequality $\|\boldsymbol{u} + \boldsymbol{v}\|_p^p \leq \|\boldsymbol{u}\|_p^p + \|\boldsymbol{v}\|_p^p$ for any $\boldsymbol{u}, \boldsymbol{v} \in \mathbb{F}^n$, we can readily obtain the following results:
\begin{align}
&(\text{dist}_p(\boldsymbol{x}^{\#},\boldsymbol{x}))^p\leq  \|\vh\|_p^p+\sigma_{2k}(\vx)_p^p\leq  \|\vh_{S_{01}}\|_p^p+\|\vh_{S_{01}^c}\|_p^p+\sigma_{2k}(\vx)_p^p\label{final1},\\
&\text{dist}(\boldsymbol{x}^{\#},\boldsymbol{x})\leq  \|\vh\|_2+\|\vx_{T_{01}^c}\|_2\overset{(a)}\leq \|\vh\|_2+\frac{\sigma_{k}(\vx)_p}{{k}^{1/p-1/2}}\leq \|\vh_{S_{01}}\|_2+\|\vh_{S_{01}^c}\|_2+\frac{\sigma_{k}(\vx)_p}{{k}^{1/p-1/2}}\label{final2},
\end{align}
where $(a)$ is based on 
\begin{equation}\label{key1}
\begin{aligned}
\|\vx_{T_{01}^c}\|_2&\leq  \sum_{j\geq 2}\|\vx_{T_j}\|_2\leq \sum_{j\geq 1}\frac{\|\vx_{T_j}\|_p}{k^{1/p-1/2}}\\
&\leq \frac{(\sum_{j\geq 1}\|\vx_{T_j}\|_p^p)^{1/p}}{k^{1/p-1/2}} \leq \frac{\|\vx_{T_0^c}\|_p}{k^{1/p-1/2}}=\frac{\sigma_k(\vx)_p}{k^{1/p-1/2}}.
\end{aligned}
\end{equation}
The inequality $\sum_{j \geq 1} \|\vx_{T_j}\|_p \leq \left( \sum_{j \geq 1} \|\vx_{T_j}\|_p^p \right)^{1/p}$ follows from the inequality
\begin{equation}\label{eqn: u_1p(a)} \|\boldsymbol{u}\|_1 \leq \|\boldsymbol{u}\|_p \quad \text{for any vector } \boldsymbol{u},\end{equation}
and define the $j$-th element of the vector $\boldsymbol{u}$ as $u_j = \|\vx_{T_j}\|_p$, for $j \geq 1$.

 Based on equations (\ref{final1}) and (\ref{final2}), we need to estimate the following quantities
 $\|\vh_{S_{01}}\|_p^p$, $\|\vh_{S_{01}^c}\|_p^p$, $\|\vh_{S_{01}}\|_2$ and $\|\vh_{S_{01}^c}\|_2$. 

\vspace{0.1cm}

{\textbf{Step 2: Dertermine the relationships between $\|\vh_{S_{01}^c}\|_p^p$ and $\|\vh_{S_{01}}\|_p^p$, and between $\|\vh_{S_{01}^c}\|_2$ and $\|\vh_{S_{01}}\|_2$}}.

We will demonstrate that
 \begin{equation}\label{eq:hs01c1}
 \begin{aligned}
&\|\vh_{S_{01}^c}\|_p^p\leq \|\vh_{S_{01}}\|_p^p+\sigma_{2k}{(\vx)}_p^p,\\
& \|\vh_{S_{01}^c}\|_2\leq  2^{1/p-1}\cdot \frac{2^{1/p-1/2}}{r^{1/p-1/2}}\cdot\|\vh_{S_{01}}\|_2+\frac{2^{1/p-1}}{r^{1/p-1/2}}\cdot \frac{\sigma_k(\vx)_p}{k^{1/p-1/2}}.
\end{aligned}
\end{equation}
Given that $\|\vx^\#\|_p^p\leq \|\vx\|_p^p$, we can deduce the following:
\[
\|\vx^\#_{S_{01}}\|_p^p+\|\vx^\#_{S_{01}^c}\|_p^p=\|\vx^\#_{S_{01}}+\vx^\#_{S_{01}^c}\|_p^p= \|\vx^\#\|_p^p\leq \|\vx\|_p^p= \|\vx_{S_{01}}\|_p^p+\|\vx_{S_{01}^c}\|_p^p,
\]
which consequently implies:
\begin{equation}\label{temp_comb11}
\begin{aligned}
\|\vh_{S_{01}^c}\|_p^p=\|\vx^\#_{S_{01}^c}\|_p^p&\leq   \|\vx_{S_{01}}\|_p^p-\|\vx^\#_{S_{01}}\|_p^p+\|\vx_{S_{01}^c}\|_p^p\\
&\leq \|\vx^\#_{S_{01}}-\vx_{S_{01}}\|_p^p+\|\vx_{S_{01}^c}\|_p^p\\
&\leq \|\vh_{S_{01}}\|_p^p+\sigma_{2k}(\vx)_p^p,
\end{aligned}
\end{equation}
where we also utilize the fact that $\|\boldsymbol{u}+\boldsymbol{v}\|_p^p\leq \|\boldsymbol{u}\|_p^p+\|\boldsymbol{v}\|_p^p$ for any $\boldsymbol{u},\boldsymbol{v}\in \mathbb{F}^n$.
Similarly, by substituting $S_{01}$ with $S_{0}$ in (\ref{temp_comb11}), we can establish the following inequality:
\begin{equation}\label{eq:bestkterm}
\|\vh_{S_0^c}\|_p^p\leq \|\vh_{S_0}\|_p^p+\sigma_{2k}{(\vx)}_p^p.
\end{equation}
Furthermore, since 
\begin{equation}\label{eqn: d}
\begin{split}
\|\vh_{S_{01}^c}\|_2\leq \sum_{j\geq 2}\|\vh_{S_j}\|_2\leq \sum_{j\geq 1}\frac{\|\vh_{S_j}\|_p}{(rk)^{1/p-1/2}}\leq \frac{(\sum_{j\geq 1}\|\vh_{S_j}\|_p^p)^{1/p}}{(rk)^{1/p-1/2}} \leq \frac{\|\vh_{S_0^c}\|_p}{(rk)^{1/p-1/2}},
\end{split}
\end{equation}
we can derive:
\begin{equation}\label{temp_comb12}
\begin{aligned}
\|\vh_{S_{01}^c}\|_2\leq&\sum_{j\geq 2}\|\vh_{S_j}\|_2\leq  \frac{\|\vh_{S_0^c}\|_p}{(rk)^{1/p-1/2}}\leq \frac{(\|\vh_{S_0}\|_p^p+\sigma_{2k}{(\vx)}_p^p)^{1/p}}{(rk)^{1/p-1/2}}\\
\leq&  2^{1/p-1}\cdot \frac{2^{1/p-1/2}}{r^{1/p-1/2}}\|\vh_{S_{01}}\|_2+\frac{2^{1/p-1}}{r^{1/p-1/2}}\cdot \frac{\sigma_k(\vx)_p}{k^{1/p-1/2}}.
\end{aligned}
\end{equation}
The last inequality above stems from  the following properties:
\begin{equation}
\label{u_p and u_2 temp}
\|\boldsymbol{u}\|_p\leq n^{1/p-1}\|\boldsymbol{u}\|_1\qquad \text{and}\qquad \|\boldsymbol{u}\|_p\leq n^{1/p-1/2}\|\boldsymbol{u}\|_2,\qquad \text {for any } \boldsymbol{u}\in \mathbb{F}^n.
\end{equation} 
Therefore, (\ref{temp_comb11}) and (\ref{temp_comb12}) substantiate the results presented in (\ref{eq:hs01c1}).

\vspace{0.1cm}

{\textbf{Step 3: Derive upper bounds for $\|\vh_{S_{01}}\|_p^p$ and $\|\vh_{S_{01}}\|_2$}}.

We shall establish that
\begin{equation}\label{eq:hs01c2}
\|\vh_{S_{01}}\|_p^p\leq \widetilde{C}_1\cdot \sigma_k(\vx)_p^p+\widetilde{D}_1\cdot k^{1-p/2}\cdot \eta^p,\quad\quad \|\vh_{S_{01}}\|_2\leq \widetilde{C}_2\cdot \frac{\sigma_k(\vx)_p}{{k}^{1/p-1/2}}+\widetilde{D}_2\cdot \eta,
\end{equation}
where 
\begin{equation}\label{c1,c2}
\widetilde{C}_1=(2+r)^{1-p/2}\cdot (\widetilde{C}_2)^p,\quad \text{with} \quad \widetilde{C}_2=\frac{U\cdot \left(\frac{2^{1/p-1}}{r^{1/p-1/2}}+1\right)}{L-U\cdot2^{1/p-1}\cdot\left(\frac{{2}}{{r}}\right)^{1/p-1/2}},
\end{equation}
and 
\begin{equation}\label{d1,d2}
\widetilde{D}_1={(2+r)}^{1-p/2} \cdot (\widetilde{D}_2)^p,\quad \text{with}\quad \widetilde{D}_2=\frac{2}{L-U\cdot2^{1/p-1}\cdot\left(\frac{{2}}{{r}}\right)^{1/p-1/2}}.
\end{equation}

We begin with the following inequalities:
\[
\||\vA\vx^{\#}|-|\vA\vx|\|_2\leq \||\vA\vx^{\#}|-|\vA\vx|-\ve\|_2+\|\ve\|_2\leq 2\eta,
\]
which implies that
\begin{equation}\label{eq:budength}
\begin{aligned}
2\eta\geq \|\abs{\vA\vx^\#}-\abs{\vA\vx}\|_2&=\|\abs{\vA\vx^\#_{S_{01}}+\vA\vx^\#_{S_{01}^c}}-\abs{\vA\vx_{T_{01}}+\vA\vx_{T_{01}^c}}\|_2\\
&\overset{(c)}{\geq} \|\abs{\vA\vx^\#_{S_{01}}}-\abs{\vA\vx_{T_{01}}}\|_2-\|\vA\vx_{T_{01}^c}\|_2-\|\vA\vx^\#_{S_{01}^c}\|_2\\
&\geq L \|\vh_{S_{01}}\|_2-\|\vA\vx_{T_{01}^c}\|_2-\|\vA\vx^\#_{S_{01}^c}\|_2,
\end{aligned}
\end{equation}
 where inequality $(c)$ is derived by repeatedly applying Lemma \ref{temp_lem}, which states that for any $\vx,\vy,\boldsymbol{s},\boldsymbol{t}\in \mathbb{F}^n$,
 \[
 \||\vx+\vy|-|\boldsymbol{s}+\boldsymbol{t}|\|_2\geq \||\boldsymbol{x}+\boldsymbol{y}|-|\boldsymbol{s}|\|_2-\|\boldsymbol{t}\|_2\geq \||\boldsymbol{x}|-|\boldsymbol{s}|\|_2-\|\boldsymbol{y}\|_2-\|\boldsymbol{t}\|_2.
 \] 
Consequently, according to (\ref{eq:budength}), we arrive at the following inequality: 
\begin{equation}\label{eq:vaxs}
\|\vA\vx_{T_{01}^c}\|_2+\|\vA\vx^\#_{S_{01}^c}\|_2+2\eta \geq L \|\vh_{S_{01}}\|_2.
\end{equation}
We will now proceed to analyze the terms $\|\vA\vx_{T_{01}^c}\|_2$ and $\|\vA\vx^\#_{S_{01}^c}\|_2$. By leveraging the results from equation (\ref{key1}) in conjunction with the phaseless bi-Lipschitz condition, we can establish the following inequality:
\begin{equation}\label{eq:atterm1}
\begin{aligned}
\|\vA\vx_{T_{01}^c}\|_2\leq \sum_{j\geq 2} \|\vA \vx_{T_j}\|_2\leq U\sum_{j\geq 2} \| \vx_{T_j}\|_2 \leq \frac{U\cdot \sigma_k(\vx)_p}{k^{1/p-1/2}}.
\end{aligned}
\end{equation}
Similarly, we can derive
\begin{equation}\label{eq:aterm2}
\begin{aligned}
\|\vA\vx^\#_{S_{01}^c}\|_2=&\|\vA\vh_{S_{01}^c}\|_2\leq \sum_{j\geq 2} \|\vA\vh_{S_j}\|_2\leq  U \sum_{j\geq 2}\|\vh_{S_j}\|_2\\
\overset{(d)}\leq &U\cdot2^{1/p-1}\cdot  \frac {2^{1/p-1/2}}{r^{1/p-1/2}}\|\vh_{S_{01}}\|_2+U\cdot\frac{2^{1/p-1}}{r^{1/p-1/2}}\cdot \frac{\sigma_k(\vx)_p}{k^{1/p-1/2}},
\end{aligned}
\end{equation}
where ($d$) follows from (\ref{temp_comb12}). 
By substituting the expressions from equations (\ref{eq:atterm1}) and (\ref{eq:aterm2}) into equation (\ref{eq:vaxs}), we derive the following inequality:
\[
\begin{split}
 L\|\vh_{S_{01}}\|_2\leq & \frac{U\cdot \sigma_k(\vx)_p}{k^{1/p-1/2}}+U\cdot2^{1/p-1}\cdot \frac {2^{1/p-1/2}}{r^{1/p-1/2}}\|\vh_{S_{01}}\|_2+U\cdot\frac{2^{1/p-1}}{r^{1/p-1/2}}\cdot \frac{\sigma_k(\vx)_p}{k^{1/p-1/2}}+2\eta\\
 \leq &U\cdot \left(\frac{2^{1/p-1}}{r^{1/p-1/2}}+1\right)\cdot \frac{\sigma_k(\vx)_p}{{k}^{1/p-1/2}}+U\cdot2^{1/p-1}\cdot\left(\frac{{2}}{{r}}\right)^{1/p-1/2}\|\vh_{S_{01}}\|_2+2\eta,
 \end{split}
\]
which implies
\begin{equation}\label{temp_comb21}
\|\vh_{S_{01}}\|_2\leq \widetilde{C}_2 \cdot \frac{\sigma_k(\vx)_p}{{k}^{1/p-1/2}}+\widetilde{D}_2\cdot \eta, 
\end{equation}
and 
\begin{equation}\label{temp_comb22}
\begin{split}
\|\vh_{S_{01}}\|_p^p\overset{(e)}\leq& {((2+r)k)}^{1-p/2} \cdot \|\vh_{S_{01}}\|_2^p\leq {((2+r)k)}^{1-p/2}\cdot ( \widetilde{C}_2 \cdot \frac{\sigma_k(\vx)_p}{{k}^{1/p-1/2}}+\widetilde{D}_2\cdot \eta)^p\\
\overset{(f)}\leq& (2+r)^{1-p/2}\cdot (\widetilde{C}_2)^p\cdot \sigma_k(\vx)_p^p+{((2+r)k)}^{1-p/2} \cdot (\widetilde{D}_2)^p\cdot \eta^p\\
=&\widetilde{C}_1\cdot \sigma_k(\vx)_p^p+\widetilde{D}_1\cdot k^{1-p/2}\cdot \eta^p.
\end{split}
\end{equation}
In this context, the inequalities ($e$) and ($f$) are derived from (\ref{u_p and u_2 temp}) and (\ref{eqn: u_1p(a)}), respectively. The constants $\widetilde{C}_1$, $\widetilde{C}_2$, $\widetilde{D}_1$, and $\widetilde{D}_2$ are defined in (\ref{c1,c2}) and (\ref{d1,d2}). Thus, the relationships delineated in (\ref{temp_comb21}) and (\ref{temp_comb22}) substantiate the conclusions articulated in (\ref{eq:hs01c2}).

\vspace{0.1cm}

{\textbf{Step 4:  
Derive the desired conclusion}}.

By substituting (\ref{eq:hs01c1}) and (\ref{eq:hs01c2}) into (\ref{final1}) and (\ref{final2}), and applying the inequality $(a^p + b^p)^{1/p} \leq 2^{1/p - 1}(a + b)$ for any $a, b \geq 0$, we derive the desired conclusions, specifically:
 \[
 \begin{aligned}
 \text{dist}_p(\boldsymbol{x}^{\#},\boldsymbol{x})&\leq (\|\vh_{S_{01}}\|_p^p+\|\vh_{S_{01}^c}\|_p^p+\sigma_{2k}(\vx)_p^p)^{1/p}\leq (2\|\vh_{S_{01}}\|_p^p+2\sigma_{k}(\vx)_p^p)^{1/p}\\
 &\leq ((2\widetilde{C}_1+2)\cdot \sigma_k(\vx)_p^p+\widetilde{D}_1\cdot k^{1-p/2}\cdot \eta^p)^{1/p}\\
&\leq 2^{1/p-1}\cdot \left((2\widetilde{C}_1+2)^{1/p}\cdot \sigma_k(\vx)_p+\widetilde{D}_1^{1/p}\cdot k^{1/p-1/2}\cdot \eta\right)\\
 &=C_1\cdot \sigma_k(\vx)_p+D_1\cdot k^{1/p-1/2}\cdot \eta
 \end{aligned}
 \]
 and 
 \[
 \begin{aligned}\text{dist}(\boldsymbol{x}^{\#},\boldsymbol{x})\leq& \|\vh_{S_{01}}\|_2+\|\vh_{S_{01}^c}\|_2+\frac{\sigma_{k}(\vx)_p}{{k}^{1/p-1/2}}\\
 \leq & \Big(1+\frac{2^{1/p-1}}{(r/2)^{1/p-1/2}}\Big)\cdot \|\vh_{S_{01}}\|_2+ \Big(1+\frac{2^{1/p-1}}{r^{1/p-1/2}}\Big)\cdot \frac{\sigma_k(\vx)_p}{k^{1/p-1/2}}\\
 \leq &\Big(1+\frac{2^{1/p-1}}{(r/2)^{1/p-1/2}}\Big)\cdot \left( \widetilde{C}_2\cdot \frac{\sigma_k(\vx)_p}{{k}^{1/p-1/2}}+\widetilde{D}_2\cdot \eta\right)+\Big(1+\frac{2^{1/p-1}}{r^{1/p-1/2}}\Big)\cdot \frac{\sigma_k(\vx)_p}{k^{1/p-1/2}}\\
 =& C_2\cdot  \frac{\sigma_k(\vx)_p}{k^{1/p-1/2}}+D_2\cdot \eta,
 \end{aligned}
 \]
 where 
 \[
C_1=2^{1/p-1}\cdot(2\widetilde{C}_1+2)^{1/p},\qquad C_2=\Big(1+\frac{2^{1/p-1}}{(r/2)^{1/p-1/2}}\Big)\cdot \widetilde{C}_2+\frac{2^{1/p-1}}{r^{1/p-1/2}}+1,
 \] 
 and
 \[
 D_1=2^{1/p-1}\cdot \widetilde{D}_1^{1/p},\qquad D_2=\Big(1+\frac{2^{1/p-1}}{(r/2)^{1/p-1/2}}\Big)\cdot \widetilde{D}_2.
 \]
 Here $\widetilde{C}_1$, $\widetilde{C}_2$, $\widetilde{D}_1$ and $\widetilde{D}_2$ are defined in (\ref{c1,c2}) and (\ref{d1,d2}). 
\end{proof}
\begin{proof}[Proof of Lemma \ref{temp_lem}]
We shall begin by establishing the right-hand side inequality of (\ref{eq:budeng}), namely:
\begin{equation}\label{eq:youbian}
\left\| \abs{\boldsymbol{x}+\boldsymbol{y}}-|\boldsymbol{z}|\right\|_2\leq \left\||\boldsymbol{x}|-|\boldsymbol{z}|\right\|_2+\|\boldsymbol{y}\|_2.
\end{equation} 
We claim that for any $a, b, c \in \mathbb{F}$,
\begin{equation}\label{claim_temp}
||a+b|-|c||\leq ||a|-|c||+|b|. 
\end{equation}
Let $x_k$, $y_k$, and $z_k$  denote the $k$th elements of $\boldsymbol{x}$, $\boldsymbol{y}$, and $\boldsymbol{z}$, respectively, where $k\in\{1,\ldots,n\}$. We can then derive:
\[
\begin{aligned}
\||\boldsymbol{x}+\boldsymbol{y}|-|\boldsymbol{z}|\|_2
=&\sqrt{\sum_{k=1}^n(|x_k+y_k|-|z_k|)^2}\leq \sqrt{\sum_{k=1}^n(||x_k|-|z_k||+|y_k|)^2}\\
=&\|||\boldsymbol{x}|-|\boldsymbol{z}||+|\boldsymbol{y}|\|_2
\leq \||\boldsymbol{x}|-|\boldsymbol{z}|\|_2+\|\boldsymbol{y}\|_2.
\end{aligned}
\]
This establishes the (\ref{eq:youbian}). It remains to prove inequality (\ref{claim_temp}).
If $\abs{a + b} \geq \abs{c}$, then we have:
\[
||a + b| - |c|| = |a + b| - |c| \leq |a| + |b| - |c| \leq ||a| - |c|| + |b|.
\]
Conversely, if $|a + b| < |c|$, then:
\[
||a + b| - |c|| = |c| - |a + b| \leq |c| - |a| + |b| \leq ||a| - |c|| + |b|.
\]
Thus, in both cases, we arrive at (\ref{claim_temp}).

Based on  (\ref{eq:youbian}), we can further deduce that for all $\vx', \vy',\vz'\in \vF^n$, it obtains
\begin{equation}\label{eq:mid1}
\left\| \abs{\boldsymbol{x}'+\boldsymbol{y}'}-|\boldsymbol{z}'|\right\|_2-\|\boldsymbol{y}'\|_2\leq \left\||\boldsymbol{x}'|-|\boldsymbol{z}'|\right\|_2.
\end{equation} 
For any $\vx,\vy,\vz\in\vF^n$, we set $\vx':=\vx+\vy, \vy':=-\vy$ and $\vz':=\vz$ in (\ref{eq:mid1}), we obtain the left side of 
(\ref{eq:budeng}).
This concludes the proof.
\end{proof}

\section{Proof of Theorem \ref{th:lip}}\label{sec: Gaussian}

We begin by introducing the following theorem, which plays a crucial role in establishing Theorem \ref{th:lip}.
\begin{theorem}(\cite[Theorem 4.1]{XXX})\label{thm: complex_beta}
Let $\boldsymbol{A}\in \vF^{m\times n}$ ($\mathbb{F}\in \{\mathbb{R},\mathbb{C}\}$) be a standard Gaussian random matrix.
For any $0<\delta<0.05$, with probability at least  $1-2\exp(-c\delta^2m)$, the matrix  $\vA\in \vF^{m\times n}$ satisfies the phaseless bi-Lipschitz condition on the set $\vF^n$
with constants $L$ and $U$  such that
\begin{equation}\label{thm: complex_beta:2024-1}
\beta_0^{\vF}\leq U/L\leq \beta_0^{\vF}+\delta
\end{equation}
provided that $m\gtrsim \log(1/\delta)\delta^{-2}n$.  
Here, $c$ denotes a universal positive constant, and $\beta_0^{\vF}$  represents the constant defined in \eqref{eq:mdlower-new}.
\end{theorem}
We shall now proceed with the proof of Theorem \ref{th:lip}.
\begin{proof}[Proof of Theorem \ref{th:lip}]
For any fixed index set $K_0\subseteq\{1,\ldots,n\}$  satisfying $\# K_0 = \lfloor (r+4)k \rfloor$, we define the set $\mathcal{X}_0$
  as follows:
\[
 \mathcal{X}_0:=\{{\vz}\in \mathbb{F}^n\ :\ {\rm supp}(\vz)\subseteq K_0\}.
\]

Applying Theorem \ref{thm: complex_beta} with $\delta = 0.01$, we can conclude that, 
with probability at least $1 - 2\exp(-2c_0m)$ for some universal constant $c_0 > 0$,
 the matrix  $\vA\in \vF^{m\times n}$ satisfies the phaseless bi-Lipschitz condition on the set  $\mathcal{X}_0$
with constants $L, U$  such that
\begin{equation}\label{thm: complex_beta:2024-1}
\beta_0^{\vF}\leq U/L\leq \beta_0^{\vF}+0.01,
\end{equation}
provided that $m\gtrsim (r+4)k$. Consequently, the matrix $\vA\in \mathbb{F}^{m\times n}$ satisfies the phaseless bi-Lipschitz condition on the set
 $\mathcal{X}:=\{\boldsymbol{x}\in \mathbb{F}^n\ :\ \|\vz\|_0\leq (r+4)k\}$ with positive constants $L$ and $U$ satisfying (\ref{thm: complex_beta:2024-1}), with probability at least 
\[
\begin{split}
1-2{n\choose \lfloor(r+4)k\rfloor}\exp(-2c_0m)\geq & 1-2{\exp\left( (r+4)k\log(en /(r+4)k)\right)}\exp(-2c_0m)\\
\geq & 1-\exp(-c_0m).
\end{split}
\]
Here, the last inequality follows from $m\geq C k \log (en/k)$, where $C$ depends on $r$. 
\end{proof}

\section{Proof of Proposition \ref{thm:bad_22}}\label{sec: 22inst_new}
 {
The proof utilizes techniques similar to those employed in Theorem 3.2 and Theorem 5.1 of \cite{best}. For  convenience, we present a restatement of Theorem 5.1 from \cite{best}, adapted to the case where $\mathbb{F} \in \{\mathbb{R}, \mathbb{C}\}$. It is worth noting that the result, originally formulated for the real case, naturally extends to complex-valued scenarios.
\begin{lem}\cite[Theorem 5.1]{best}
\label{lem: lem_temp}
Let $\vA \in \mathbb{F}^{m \times n}$ be a matrix. Define its null space as:
\begin{equation}
\label{eqn: Null_A}
\mathcal{N}(\vA) := \{\boldsymbol{\eta} \in \mathbb{F}^n : \vA\boldsymbol{\eta} = \boldsymbol{0}\}.
\end{equation}
 Suppose that: for all $\boldsymbol{\eta} \in \mathcal{N}(\vA)$ and all subsets $T \subset \{1,\ldots,n\}$ with cardinality $\# T  \leq 2k$, 
\begin{equation}
 \label{eqn: conclusion_eta_temp} 
\|\boldsymbol{\eta}\|_2 \leq C_0\cdot  \|\boldsymbol{\eta}_{T^c}\|_2,
\end{equation}
where $C_0$  is a positive absolute constant. Then we have $m \gtrsim n$.
\end{lem}
\begin{proof}[Proof of Proposition \ref{thm:bad_22}]
Proving  (\ref{eqn: conclusion_eta_temp}) of Lemma \ref{lem: lem_temp} is sufficient to establish the desired conclusion.
 For any  $\veta\in \mathcal{N}(\vA)$ defined in (\ref{eqn: Null_A}), let $T_1$ be the index set of $\veta$'s $k$ largest magnitude elements, and $T_2$  be that of $\veta_{T_1^c}$'s $k$ largest magnitude elements. For convenience, define $\veta_1:=\veta_{T_1}$, $\veta_2:=\veta_{T_2}$, and $\veta_3:=\veta-\veta_{1}-\veta_{2}$.\\
\indent Since $\veta\in \mathcal{N}(\vA)$ and $\veta=\veta_1+\veta_2+\veta_3$, we have $\vA(-\veta_1)=\vA(\veta_2+\veta_3)$, which implies  $|\vA(\veta_1)|=|\vA(\veta_2+\veta_3)|$. Thus, according to phaseless $(2,2)$-instance optimality in  
(\ref{eqn: uniform_22_upper}), we have
\[
  \mathrm{dist}(\Delta(|\vA(\veta_2+\veta_3)|),\veta_1)_2=  \mathrm{dist}(\Delta(|\vA(\veta_1)|),\veta_1)_2\leq c_0\cdot  \sigma_k(\veta_1)_2=0.
\]
Take $c_{\veta}\in \mathbb{F}$ such that $|c_{\veta}|=1$ and 
$
c_{\veta}\cdot \veta_1=\Delta(|\vA(\veta_2+\veta_3)|).
$
Then, also based on $(2,2)$-instance optimality, we have 
\[
\begin{split}
\|\veta\|_2=&\|\veta_1+\veta_2+\veta_3\|_2\overset{(a)}=\dist(c_{\veta}\cdot \veta_1,\veta_2+\veta_3)=\text{dist}(\Delta(|\vA(\veta_2+\veta_3)|),\veta_2+\veta_3)\\
\leq & c_0\cdot  \sigma_k(\veta_2+\veta_3)_2=c_0\cdot  \|\veta_3\|_2\leq c_0\cdot \sigma_{2k}(\veta)_2\leq c_0\cdot \|\veta_{T^c}\|_2,
\end{split}
\]
for all $T \subset \{1,\ldots,n\}$ with $\#T \leq 2k$. Here ($a$) follows from the disjoint support property of $\veta_2+\veta_3$ and $\veta_1$. Therefore, (\ref{eqn: conclusion_eta_temp}) holds with $C_0:=c_0$ and the proof is completed. 
\end{proof} }

\section{Proof of Theorem \ref{th:22}}\label{sec: 22inst2}
We now turn our attention to the  phaseless $(2,2)$-instance optimality in probability. 
Lemmas \ref{Gaussian_LQ} and \ref{construction of x} are derived from Lemma 2.1 and Theorem 4.1 in \cite{L1222}, respectively. {Although Lemma 2.1 and Theorem 4.1 in \cite{L1222} are originally formulated for the real-valued Bernoulli case, they can be readily extended to standard Gaussian measurements for $\mathbb{F} \in \{\mathbb{R}, \mathbb{C}\}$, as outlined in Section 5 of \cite{L1222}. For the sake of brevity, we omit the proof of this extension.}

\begin{lem}\cite[Lemma 2.1]{L1222}\label{Gaussian_LQ}
Suppose that $\vA\in \mathbb{F}^{m\times n}$ is a standard Gaussian random matrix. For each fixed $\vz\in \mathbb{F}^n$, then we have 
\begin{equation}\label{z_temp}
\frac{1}{\sqrt{m}}\|\vA\vz\|_2\leq \sqrt{\frac{3}{2}}\|\vz\|_2\qquad \text{and}\qquad \frac{1}{\sqrt{m}}\|\vA\vz\|_\infty\leq \frac{1}{\sqrt{\log(en/m)}}\|\vz\|_2
\end{equation}
with probability at least $1-2\exp(-\widetilde{c}_0m)-2m\exp(-\widetilde{c}_1m/\log(en/m))$. Here $\widetilde{c}_0$ and $\widetilde{c}_1$ are universal positive constants. 
\end{lem}

\begin{lem}\cite[Theorem 4.1]{L1222}\label{construction of x}
Assume that $m$ and $n$ are positive integers with $n \geq (\log 6)^2 m$. Let $\vA \in \mathbb{F}^{m \times n}$ be a standard Gaussian random matrix. For each vector $\vy\in \mathbb{F}^m$ satisfying $\|\vy\|_\infty \leq \frac{1}{\sqrt{m}}$ and $\|\vy\|_2 \leq \sqrt{\frac{\log(en/m)}{m}}$, with probability  at least $1 - \widetilde{C}_1 \exp(-\widetilde{c}_2 m) - \exp(-\sqrt{mn})$, there exists a vector $\widetilde{\vz} \in \mathbb{F}^n$ such that $\vy = \frac{1}{\sqrt{m}} \vA \widetilde{\vz}$, and the conditions $\|\widetilde{\vz}\|_1 \leq C_1$ and $\|\widetilde{\vz}\|_2 \leq C_2 \frac{1}{\sqrt{k}}$ hold for all $k \leq \frac{a m}{\log(en/m)}$. Here, $a$, $C_1$, $C_2$, $\widetilde{C}_1$, and $\widetilde{c}_2$ are universal positive constants.

\end{lem}

\begin{proof}[Proof of Theorem \ref{th:22}]
Assume that $C \geq \max\{1, C_0, 1/a\}$ is a universal constant, where $C_0$ and $a$ are universal positive constants derived from Corollary \ref{(2,1) and (1,1)} and Lemma \ref{construction of x}, respectively.
Given the condition that $k$ satisfying $Ck\log(en/k) \leq m$, we shall further partition the proof into two cases: Case 1: $k > \frac{en}{16C^2}$,  Case 2: $k \leq \frac{en}{16C^2}$.

{\textbf{{Case 1: }$k > \frac{en}{16C^2}$}}. Since $C_0 k \log(en/k) \leq C k \log(en/k) \leq m$, we can directly invoke Corollary \ref{(2,1) and (1,1)} to conclude that, with probability at least $1 - \exp(-c_0 m)$ for some universal positive constant $c_0$:
\[
\text{dist}(\Delta_{1,0}(|\vA \vx_0|), \vx_0) \lesssim \frac{\sigma_{k}(\vx_0)_1}{\sqrt{k}} {<} \frac{\sqrt{n} \sigma_{k}(\vx_0)_2}{\sqrt{en/(16C^2)}} \lesssim \sigma_{k}(\vx_0)_2.
\]

{\textbf{{Case 2: }$k \leq \frac{en}{16C^2}$}}. Let $k_0$ defined as 
\begin{equation}\label{k_0}
k_0:=\max\left\{s\in {\mathbb Z} \ :\ 1\leq s \leq \frac{en}{16C^2}, \quad \text{and}\quad C s \log\left({en}/{s}\right) \leq m\right\}.
\end{equation}
 It directly implies that $k_0\leq m$ and $k\leq k_0$. Denote  $T_0$ as the index set that contains the largest $k_0$ elements of $\vx_0$  in magnitude. If $(\vx_0)_{T_0^c} = \boldsymbol{0}$, based on $C_0 k_0 \log(en/k_0) \leq C k_0 \log(en/k_0) \leq m$, we can invoke Corollary \ref{(2,1) and (1,1)} to demonstrate that, with probability at least $1 - \exp(-c_0 m)$:
\[\text{dist}(\Delta_{1,0}(|\vA \vx_0|), \vx_0) \lesssim \frac{\sigma_{k_0}(\vx_0)_1}{\sqrt{k_0}} = 0 \leq \sigma_{k}(\vx_0)_2.\]
Next, we consider the scenario where $(\vx_0)_{T_0^c} \neq \boldsymbol{0}$. The subsequent proof will be divided into two subcases: $n \leq (\log 6)^2 m$ and $n \geq (\log 6)^2 m$.

{\textbf{{Subcase i: }$k \leq \frac{en}{16C^2}$ and ${n \leq (\log 6)^2 m}$}}.  Let $\alpha_0{<\frac{e}{16C^2}}$ be a sufficiently small universal positive constant such that
$C \cdot \alpha_0 \cdot {\log(e/\alpha_0)}\leq 1/{(\log 6)^2} .$
It then follows that
$C \cdot (\alpha_0 n) \cdot \log\left(\frac{en}{\alpha_0 n}\right) \leq \frac{n}{(\log 6)^2} \leq m,$
which consequently implies that $k_0 \geq \alpha_0 n$ by the definition of $k_0$ in (\ref{k_0}). Given that $C_0 k_0 \log\left(\frac{en}{k_0}\right) \leq C k_0 \log\left(\frac{en}{k_0}\right) \leq m$, we can also apply Corollary \ref{(2,1) and (1,1)} to ascertain that, with  probability  at least $1 - \exp(-c_0 m)$:
\[\text{dist}(\Delta_{1,0}(|\vA \vx_0|), \vx_0) \lesssim \frac{\sigma_{k_0}(\vx_0)_1}{\sqrt{k_0}} \overset{(a)}{\leq} \frac{\sqrt{n} \sigma_{k_0}(\vx_0)_2}{\sqrt{\alpha_0 n}} \lesssim \sigma_{k_0}(\vx_0)_2 \leq \sigma_{k}(\vx_0)_2,\]
 as $k \leq k_0$. Here, $(a)$ follows from the condition $k_0 \geq \alpha_0 n$.

{\textbf{{Subcase ii: }$k \leq \frac{en}{16C^2}$ and ${{n \geq (\log 6)^2 m}}$}}.
We claim that  there exists $\widetilde{\vx}\in \mathbb{F}^n$ such that 
\begin{equation}\label{key_temp}
\small
\boldsymbol{A}(\vx_0)_{T_0^c}=\boldsymbol{A}\widetilde{\vx},\, \|\widetilde{\vx}\|_1\leq C_1\sqrt{\frac{3m}{2\log(en/m)}}{{\|(\vx_0)_{T_0^c}\|_2}},\text {and } \|\widetilde{\vx}\|_2\leq C_2\sqrt{\frac{3m}{2k_0\log(en/m)}}{{\|(\vx_0)_{T_0^c}\|_2}},
\end{equation}
where $C_1$ and $C_2$ are  universal positive constants from Lemma \ref{construction of x}. 
Based on Corollary \ref{(2,1) and (1,1)}  with  $|\vA \vx_0|=|\vA( (\vx_0)_{T_0}+\widetilde{\vx})|$,  we have 
\[
\text{dist}(\Delta_{1,0}(|\vA\vx_0|), (\vx_0)_{T_0}+\widetilde{\vx})\lesssim \frac{\sigma_{k_0}( (\vx_0)_{T_0}+\widetilde{\vx})_1}{\sqrt{k}_0}\leq \frac{\|\widetilde{\vx}\|_1}{\sqrt{k_0}}\overset{(c)}\lesssim \sqrt{\frac{m}{k_0\log(en/m)}}{{\|(\vx_0)_{T_0^c}\|_2}}.
\]
Here $(c)$ is derived from (\ref{key_temp}). Therefore, since $k\leq k_0$, it obtains that 
\[
\begin{split}
\text{dist}(\Delta_{1,0}(|\vA\vx_0|), \vx_0)
\leq & \text{dist}(\Delta_{1,0}(|\vA\vx_0|), (\vx_0)_{T_0}+\widetilde{\vx})+\|(\vx_0)_{T_0}-\vx_0\|_2+\|\widetilde{\vx}\|_2\\
\lesssim &  \left( \sqrt{\frac{m}{k_0\log(en/m)}}+1\right)\sigma_{k_0}(\vx_0)_2\lesssim \sigma_{k_0}(\vx_0)_2\leq \sigma_{k}(\vx_0)_2.
\end{split}
\]
The second line follows from (\ref{key_temp}) and  the inequality 
\begin{equation}\label{eqn: k_0}
{
m\lesssim k_0 \log\left({en}/{m}\right),}
\end{equation} the proof of which will be deferred until the end of the argument. Then  the proof is complete.

We now proceed {to prove (\ref{key_temp})} and (\ref{eqn: k_0}).

Prior to the construction of $\widetilde{\vx}$ to meet (\ref{key_temp}), we define $\vz$ as
\[
\vz: = \frac{\sqrt{2}}{\sqrt{3} \|(\vx_0)_{T_0^c}\|_2} \cdot \sqrt{\frac{\log(en/m)}{m}} \cdot (\vx_0)_{T_0^c},
\]
and let $\vy: = \frac{1}{\sqrt{m}} \vA \vz$. The construction of $\widetilde{\vx}$ in  will be derived from $\vy$. To begin with, we shall explore the properties of $\vy$.
 According to (\ref{z_temp}) in Lemma \ref{Gaussian_LQ}, with probability at least $1-2\exp(-\widetilde{c}_0m)-2m\exp(-\widetilde{c}_1m/\log(en/m))$,  we can readily conclude that
  \[
  \|\vy\|_\infty\leq 1/\sqrt{m}\quad \text{and}\quad\|\vy\|_2\leq \sqrt{\log(en/m)/m}.
  \] 
  Here $\widetilde{c}_0$ and $\widetilde{c}_1$ are universal positive constants. Furthermore, we can obtain the following result:
\[
k_0 \log\left(\frac{en}{m}\right) \leq k_0 \log\left(\frac{en}{k_0}\right) \leq \frac{m}{C} \leq am,
\]
given that $k_0 \leq m$ and $C \geq{1}/{a}$. Consequently, the conditions stipulated in Lemma \ref{construction of x} are satisfied, allowing us to construct $\widetilde{\vz}$ such that
$\vy = \frac{1}{\sqrt{m}} \vA \widetilde{\vz}$ with
\begin{equation}\label{C_1 and C_2}
\|\widetilde{\vz}\|_1\leq C_1\quad \text{and}\quad \|\widetilde{\vz}\|_2\leq C_2\frac{1}{\sqrt{k_0}},
\end{equation}
with probability  at least $1 - \widetilde{C}_1 \exp(-\widetilde{c}_2 m) - \exp(-\sqrt{mn})$.  Here $C_1$, $C_2$, $\widetilde{C}_1$, and $\widetilde{c}_2$ are also universal positive constants. We can now define $\widetilde{\vx}$ as follows:
\[
\widetilde{\vx} := \sqrt{\frac{3m}{2\log(en/m)}} \cdot {\|(\vx_0)_{T_0^c}\|_2}\cdot \widetilde{\vz}.
\]
By invoking the results from (\ref{C_1 and C_2}) and $\boldsymbol{A}(\vx_0)_{T_0^c}=\boldsymbol{A}\widetilde{\vx}$, we can readily derive the conclusion presented in (\ref{key_temp}).

The sole remaining task is to establish the validity of (\ref{eqn: k_0}). If $k_0=\lfloor \frac{en}{16C^2}\rfloor$, given that $n \geq (\log 6)^2\cdot m$, we derive the following inequality:
\[
\begin{aligned}
k_0\log(en/m)\geq k_0\log(e\cdot (\log 6)^2)\geq \frac{en}{32C^2}\cdot \log(e\cdot(\log 6)^2)\geq \frac{e(\log 6)^2\log(e\cdot (\log 6)^2)}{32C^2}m.
\end{aligned}
\] 
Thus, we can conclude that (\ref{eqn: k_0}) holds true. We will now consider the case where $k_0 < \lfloor \frac{en}{16C^2} \rfloor$.  {In this context, $s_0 := k_0 + 1$ fails to satisfy the condition $Cs_0 \log(en/s_0) \leq m$, since $k_0$ is the largest integer $s$ such that both $s \leq \frac{en}{16C^2}$ and $Cs \log(en/s) \leq m$ hold true. } 
Then we can obtain the following inequality:
\[
m \leq C(k_0 + 1) \log\left({en}/{(k_0 + 1)}\right),
\]
which consequently leads to the result:
\[m \leq 2C k_0 \log\left({en}/{k_0}\right).\]
We can then deduce that
\[
\begin{aligned}
\frac{k_0\log(en/m)}{m}&\geq \frac{\log(en/(2Ck_0))-\log(\log(en/k_0))}{2C\log(en/k_0)} \overset{(d)}\geq  \frac{1}{2C}\cdot (\frac{1}{2}-0.37)\geq \frac{0.05}{C},
\end{aligned}
\]
thereby leading us to the conclusion stated in (\ref{eqn: k_0}). Here $(d)$ follows from the relationship: {
\begin{equation}\label{two_ineq}
\log(en/(2Ck_0))/\log(en/k_0)\geq 1/2\quad \text{and}\quad  \log(\log(en/k_0))/\log(en/k_0)\leq 0.37.
 \end{equation}}
The first inequality in (\ref{two_ineq}) is derived from $(en/(2Ck_0))^2 \geq en/k_0$, which holds because $k_0 < \frac{en}{16C^2} \leq \frac{en}{4C^2}$. The second inequality in (\ref{two_ineq}) follows from $k_0 < \frac{en}{16C^2} \leq \frac{en}{16}$ and 
\[
 \log(\log(en/k_0))/\log(en/k_0)\leq \log(\log(16))/\log(16)\leq 0.37.
\]
\end{proof}

\end{document}